\documentclass{article}
\usepackage[affil-it]{authblk}
\usepackage[utf8]{inputenc}
\usepackage{amsmath, amsfonts, amsthm, amssymb, latexsym}
\usepackage{geometry}
\usepackage{enumerate}
\usepackage{multirow}
\usepackage{rotating}
\usepackage{graphicx,color}
\usepackage{float}
\usepackage{mathrsfs}

\usepackage{xargs}
\usepackage[pdftex,dvipsnames]{xcolor}

\usepackage{hyperref}
\hypersetup{colorlinks=true, pdfstartview=FitV,linkcolor=blue!70!black,citecolor=red!70!black, urlcolor=green!60!black}
\definecolor{labelkey}{rgb}{0.6,0,0}

\usepackage[colorinlistoftodos,prependcaption,textsize=small]{todonotes}
\newcommandx{\change}[2][1=]{\todo[#1]{#2}}
\newcommandx{\unsure}[2][1=]{\todo[linecolor=red,backgroundcolor=red!25,bordercolor=red,#1]{#2}}
\newcommandx{\rmk}[2][1=]{\todo[linecolor=blue,backgroundcolor=blue!25,bordercolor=blue,#1]{#2}}
\newcommandx{\info}[2][1=]{\todo[linecolor=OliveGreen,backgroundcolor=OliveGreen!25,bordercolor=OliveGreen,#1]{#2}}
\newcommandx{\improvement}[2][1=]{\todo[linecolor=Plum,backgroundcolor=Plum!25,bordercolor=Plum,#1]{#2}}
\newcommandx{\thiswillnotshow}[2][1=]{\todo[disable,#1]{#2}}

\newtheorem{thm}{Theorem}[section]

\newtheorem{define}[thm]{Definition}
\newtheorem{prop}[thm]{Proposition}

\theoremstyle{definition}

\theoremstyle{remark}

\title{An inverse problem for the space-time fractional Schr\"odinger equation on closed manifolds}

\author{Li Li \thanks{lili19940301@mail.tsinghua.edu.cn}}
\affil{Yau Mathematical Sciences Center, Tsinghua University, Beijing, China}

\date{}

\begin{document}
	
	\maketitle
	
	\noindent \textbf{ABSTRACT.}\, We formulate an inverse problem for an uncoupled space-time fractional Schr\"odinger equation on closed manifolds. Our main goal is to determine the fractional powers and the Riemannian metric (up to an isometry) simultaneously from the knowledge of the associated source-to-solution map. Our argument relies on the asymptotic behavior of  Mittag–Leffler functions, Weyl's law for the eigenvalues of the Laplace-Beltrami operator, the unique continuation property of the space-fractional operator and the disjoint data metric determination result for the wave equation. We also provide a probabilistic formulation of our inverse problem.
	
	\section{Introduction}
	The study of fractional quantum mechanics at least dates back to the late 1990s. In a series of papers, Laskin replaced the classical Laplacian (in space variables) in the Schr\"odinger equation by the fractional Laplacian, generalizing the Feynman path integral to the L\'evy one (see \cite{laskin2000fractional, laskin2002fractional}). Later, models involving (time-fractional) Caputo derivatives have been introduced in fractional quantum mechanics. We refer readers to \cite{naber2004time} for the Naber model and \cite{achar2013time} for the Achar-Yale-Hanneken model.
	
	More recently, models combining both space-fractional and time-fractional features have been introduced to describe quantum systems (see \cite{dong2008space, hislop2024edge}). In this paper, we are interested in such a kind of space-time fractional Schr\"odinger operators, and we will formulate and study a corresponding inverse problem. More precisely, we consider the following initial value problem for an uncoupled space-time fractional Schr\"odinger equation
	\begin{equation}\label{fracSchro}
		\left\{
		\begin{aligned}
			i\partial^\alpha_{t}u +(-\Delta_g)^\beta u &= f,\quad \,\,\, \mathrm{in}\,\, M\times (0, \infty),\\
			u(0)&= 0,\quad \,\,\,\mathrm{in}\,\, M.\\
		\end{aligned}
		\right.
	\end{equation}
	Here $\alpha, \beta\in (0, 1)$ are constant fractional powers.
	$(M, g)$ is a closed Riemannian manifold, and $(-\Delta_g)$ is the Laplace-Beltrami operator.
	
	To formulate our inverse problem, we need to define our measurement map. More precisely, for given nonempty open subsets $W_1, W_2 \subset M$, we define the source-to-solution map
	\begin{equation}\label{stosol}
		L^{\alpha, \beta}_{W_1, W_2}: f\to u_f|_{W_2\times (0, \infty)}, \qquad \mathrm{supp} (f)\subset W_1\times (0, \infty),
	\end{equation}
	where $u_f$ is the solution of (\ref{fracSchro}) corresponding to the source $f$. We will show that $L^{\alpha, \beta}_{W_1, W_2}$ is well-defined for sufficiently smooth $f$.
	
	Our goal is to determine $\alpha, \beta$ and the metric $g$ (up to an isometry) simultaneously from the knowledge of $L^{\alpha, \beta}_{W_1, W_2}$. The following theorem is our main result.
	
	\begin{thm}\label{Th}
		Suppose $(M, g)$, $(M, \tilde{g})$ are smooth connected closed Riemannian manifolds. Suppose $0< \alpha, \tilde{\alpha}< 1$ and $0< \beta,  \tilde{\beta}< 1$. Let $W_1, W_2\subset M$ be nonempty disjoint open subsets. Suppose $W_1$ satisfies the spectral bound condition (see (\ref{specbdcond}) below in Subsection 2.1) and $W_2$ has smooth boundary. Let $L^{\alpha, \beta}_{W_1, W_2}, \tilde{L}^{\tilde{\alpha}, \tilde{\beta}}_{W_1, W_2}$ be the source-to-solution maps corresponding to $g, \tilde{g}$.
		Suppose {$g= \tilde{g}$ in $W_1\cup W_2$} and
		\begin{equation}\label{stosoleq}
			L^{\alpha, \beta}_{W_1, W_2} f= \tilde{L}^{\tilde{\alpha}, \tilde{\beta}}_{W_1, W_2} f,\qquad
			f\in C^2_c((0, \infty); L^2(W_1)).
		\end{equation}
		Then $\alpha= \tilde{\alpha}$, $\beta=  \tilde{\beta}$, and $(M, g)$ and $(M, \tilde{g})$ are isometric Riemannian manifolds.
	\end{thm}
	
	\subsection{Connection with earlier literature}
	
	Inverse spectral theory and the associated metric determination problem for hyperbolic equations have been extensively studied so far. The approach mainly relies on the boundary control method
	(see \cite{belishev1992reconstruction}) and its variant. We refer readers to the textbook \cite{kachalov2001inverse} for a detailed study of this topic. We also refer readers to \cite{lassas2014inverse, krupchyk2008inverse, helin2018correlation, li2024inverse, lassas2023disjoint} for metric determination results on manifolds with or without boundary.
	
	There is also abundant literature available on inverse problems for fractional operators. The rigorous mathematical study of inverse problems for time-fractional but space-local equations at least dates back the late 2000s (see for instance, \cite{cheng2009uniqueness, kian2021uniqueness}). The study of Calder\'on type inverse problems for space-fractional equations
	was initiated in \cite{ghosh2020calderon}. We refer readers to \cite{ghosh2020uniqueness, covi2022higher, covi2022uniqueness, lai2023inverse} and the references inside for extended results for the fractional Calder\'on problem and its variants. In particular, we mention that the Calder\'on type inverse problem for the (coupled or uncoupled) space-time fractional parabolic equations has been studied in \cite{lai2020calderon, li2022inverse}. We also mention that the determination result for exponents in the space-time fractional diffusion equation (in 1-dim case) has been obtained in \cite{tatar2016simultaneous}.  
	
	In addition, there are a few recent works on the metric determination for fractional equations. One seminal work is \cite{feizmohammadi2021fractional}, where the authors reduced the fractional elliptic problem to the classical hyperbolic problem by taking advantage of the semigroup definition of the fractional operator and the transmutation formula. Concerning the model in this paper, a more related work is \cite{helin2020inverse}, where the authors determined the metric up to an isometry from the knowledge of the source-to-solution map associated with the space–time fractional diffusion equation.
	
	In this paper, we seek a simultaneous determination of the exponents and the metric in the space-time fractional Schr\"odinger equation, and the proof of the main theorem involves a combination of different kinds of techniques. More precisely, our determination of the time-fractional power $\beta$ relies on the asymptotic behavior of Mittag–Leffler functions and the unique continuation property of the space-fractional operator;
	Our determination of the space-fractional power $\alpha$ and the metric $g$ relies on Weyl's law for the eigenvalues and the disjoint data metric determination result for the wave equation established in \cite{lassas2023disjoint}.
	
	\subsection{Organization}
	The rest of this paper is organized in the following way. In
	Section 2, we will summarize the preliminary knowledge. In Section 3, we will study the well-posedness of (\ref{fracSchro}) for sufficiently smooth sources. In Section 4, we will prove the main theorem. In Section 5, we will provide a probabilistic formulation
	of our inverse problem.
	\medskip
	
	\noindent \textbf{Acknowledgments.} The author would like to thank Professor Katya Krupchyk for suggesting the problem and for helpful discussions.
	
	\section{Preliminaries}
	Throughout this paper, the space dimension $n\geq 2$. We use $\langle\cdot, \cdot\rangle$ to denote the $L^2$-inner product, and we use $\alpha, \beta\in (0, 1)$ to denote constant fractional powers.
	
	\subsection{Laplace-Beltrami operator and Sobolev spaces}
	Let $(M, g)$ be a closed manifold. Let $0= \lambda_0< \lambda_1\leq \lambda_2<\cdots\to +\infty$ be the eigenvalues of the Laplace-Beltrami operator $-\Delta_g$. Then there exists an orthonormal basis of $L^2(M)$ consisting of the eigenfunctions $\varphi_k$ corresponding to $\lambda_k$. Recall that for $r\geq 0$,
	$$H^r(M) =  \{u \in L^2(M): \sum_{k=1}^\infty \lambda_k^{r} |\langle u, \varphi_k \rangle|^2 < +\infty\}.$$
	The fractional operator $(-\Delta_g)^r$ is defined by
	\begin{equation}\label{specfrac}
		(-\Delta_g)^r u:=\sum_{k=1}^\infty \lambda_k^r \langle u, \varphi_k \rangle\varphi_k,\qquad u\in H^{2r}(M).
	\end{equation}
	
	We say that a nonempty open subset $W\subset M$ satisfies the spectral bound condition provided that there exists a constant $C_0> 0$ such that for any basis of normalized eigenfunctions $\{\varphi_k\}$ of the Laplace-Beltrami operator, the estimate
	\begin{equation}\label{specbdcond}
		1\leq C_0||\varphi_k||_{L^2(W)}
	\end{equation}
	holds for all $k\in \mathbb{N}$. We remark that this condition holds true for an Anosov surface, regardless of the specific choice of $W$. We refer readers to \cite{lassas2023disjoint} for more details of the Anosov property.
	
	We will later use the following unique continuation property on closed manifolds.
	
	\begin{prop}\label{UCP}
		Let $r\in (0, \infty)\setminus \mathbb{N}$. Let $W\subset M$ be nonempty and open. Suppose that $u\in H^{2r}(M)$ satisfies
		$$(-\Delta_g)^r u= u= 0$$ in $W$. Then $u= 0$ in $M$.
	\end{prop}
	We remark that the unique continuation property of the fractional Laplacian in $\mathbb{R}^n$ was first proved in \cite{ghosh2020calderon} based on the Carleman estimates in \cite{ruland2015unique}.
	We also remark that Proposition \ref{UCP} can be viewed as a special case of the more general entanglement principle (see Theorem 1.8 in \cite{feizmohammadi2024calder}).

	\subsection{Caputo derivatives and Mittag–Leffler functions}
	
	For a sufficiently regular function $u(t)$, its Caputo derivative is given by
	\begin{equation}\label{capdef}
		\partial_t^\alpha u(t)= \frac{1}{\Gamma(1-\alpha)}\int^t_0\frac{u'(\tau)}{(t-\tau)^\alpha}\,\mathrm{d}\tau.
	\end{equation}
	If $u(0)= 0$, then we have the following Laplace transform formula
	$$\mathcal{L}(\partial_t^\alpha u)(s)= s^\alpha \mathcal{L}u (s).$$
	
	For $a, b> 0$, we have the (two-parameter) Mittag–Leffler function
	$$E_{a,b}(z):= \sum^\infty_{k=0}\frac{z^k}{\Gamma(ka+b)},$$
	which is an entire function generalizing the exponential function $e^z= E_{1,1}(z)$. Here $\Gamma$ is the standard Gamma function.
	For fixed $a$, we define 
	$$G_\xi(z):= E_{a, 1}(-\xi z^a)$$
	for each $\xi\in \mathbb{C}$, which is holomorphic on $\{z\in \mathbb{C}: \mathrm{Re}\,z> 0\}$.
	Based on the relation $\Gamma(z+1)= z\Gamma(z)$, it is easy to verify that 
	\begin{equation}\label{Ea1deri}
		\frac{\mathrm{d}}{\mathrm{d}z}G_\xi(z)= -\xi z^{a-1}E_{a, a}(-\xi z^a).
	\end{equation}
	
	The following proposition corresponds to Theorem 1.4 and Theorem 1.6 in \cite{podlubny1998fractional}.
	\begin{prop}\label{BddAsym}
		Suppose $0< a< 2$, $b> 0$ and $\pi a/2< \mu< \min\{\pi, \pi a\}$. Then we have
		$$|E_{a,b}(z)|\leq \frac{C}{1+ |z|},\qquad \mu\leq |\arg(z)|\leq \pi$$
		for some positive constant $C$. We also have
		$$E_{a,1}(z)= -\frac{z^{-1}}{\Gamma(1-\alpha)}+ \mathrm{O}(|z|^{-2}),\qquad \mu\leq |\arg(z)|\leq \pi$$
		as $|z|\to \infty$.
	\end{prop}
	In particular, since $0< \alpha< 1$, we can choose $a= \alpha$, $\pi\alpha/2 < \mu< \min\{\pi/2, \pi\alpha\}$ in the proposition above, which implies $E_{\alpha, \alpha}$ and $E_{\alpha, 1}$
	are bounded on the imaginary axis, and $E_{\alpha, 1}$ has the above asymptotic expansion on the imaginary axis.
	
	\section{Well-posedness}
	Our goal here is to establish the well-posedness of (\ref{fracSchro}).
	We are not going to pursue the optimal regularity since this is not our main concern.
	
	\begin{prop}\label{wellpose}
		For $f\in C^2_c((0, \infty); L^2(M))$, there exists a unique solution of (\ref{fracSchro})
		satisfying
		$$u\in C^1([0, \infty); H^{2\beta}(M))\cap L^\infty(0, \infty; H^{2\beta}(M)),\quad 
		\partial_t u\in L^\infty(0, \infty; H^{2\beta}(M)).$$
		
	\end{prop}
	\begin{proof}
		We will construct the solution in the series form of $\sum^\infty_{k= 0}u_k(t)\varphi_k$. To determine the Fourier coefficients $u_k$, we need to solve the fractional ODE
		\begin{equation}\label{fracODE}
			\left\{
			\begin{aligned}
				(i\partial_t^\alpha + \lambda^\beta_k) u_k&= f_k,\\
				u_k(0)&= 0,\\
			\end{aligned}
			\right.
		\end{equation}
		where $f_k(t)= \langle f(t), \varphi_k\rangle\in C^2_c((0, \infty))$.
		Based on the convolution property of the Laplace transform, this fractional ODE has the solution
		\begin{equation}\label{Kfconvolute}
			u_k = K_k* f_k
		\end{equation}
		(see (7) in \cite{ashurov2022time}), where $K_k(t)$ satisfies the Laplace transform
		\begin{equation}\label{KkLaplace}
			\mathcal{L}K_k(s)= \frac{1}{is^\alpha+ \lambda^\beta_k}.
		\end{equation}
		Based on the formula (4.9.1) in \cite{gorenflo2020mittag}, we know that
		\begin{equation}\label{Kk}
			K_k(t)= -it^{\alpha-1}E_{\alpha, \alpha}(i\lambda^\beta_k t^\alpha).
		\end{equation}
		
		The $j^{th}$ derivatives of $u_k$ (with respect to time) satisfies
		$$u^{(j)}_k = K_k* f^{(j)}_k,\qquad j= 0, 1.$$
		Then we have
		$$\lambda^\beta_ku^{(j)}_k(t)=
		\int^t_0 -i\lambda^\beta_k(t-\tau)^{\alpha-1}
		E_{\alpha, \alpha}(i\lambda^\beta_k (t-\tau)^\alpha)f^{(j)}_k(\tau)\,d\tau.$$
		Using $E_{a, 1}(0)= 1$, $f^{(j)}_k(0)= 0$ and (\ref{Ea1deri}), we integrate by parts to obtain
		\begin{equation}\label{ibp}
			\lambda^\beta_ku^{(j)}_k(t)=
			f^{(j)}_k(t)- \int^t_0E_{\alpha, 1}(i\lambda^\beta_k (t-\tau)^\alpha)f^{(j+1)}_k(\tau)d\tau.
		\end{equation}
		Then we have
		$$\lambda^\beta_k|u^{(j)}_k(t)|
		\leq |f^{(j)}_k(t)|+ \int^t_0|E_{\alpha, 1}(i\lambda^\beta_k (t-\tau)^\alpha)f^{(j+1)}_k(\tau)|d\tau$$
		$$\leq \int^t_0(1+|E_{\alpha, 1}(i\lambda^\beta_k (t-\tau)^\alpha)|)|f^{(j+1)}_k(\tau)|d\tau.$$
		
		Suppose $\mathrm{supp}\,f\subset (0, T_0)\times M$ for some 
		$T_0> 0$.
		By the boundedness of $E_{\alpha, 1}$ on the imaginary axis and Cauchy-Schwarz inequality, we have
		$$\lambda^{2\beta}_k|u^{(j)}_k(t)|^2
		\leq \int^{\min\{t, T_0\}}_0 (1+ |E_{\alpha, 1}(i\lambda^\beta_k (t-\tau)^\alpha)|)^2 d\tau
		\int^{\min\{t, T_0\}}_0|f^{(j+1)}_k(\tau)|^2 d\tau$$
		$$\leq C_\alpha T_0 \int^{T_0}_0|f^{(j+1)}_k(\tau)|^2 d\tau.$$
		By the dominated convergence theorem,
		$$\sum^\infty_{k= N}\lambda_k^{2\beta} |u^{(j)}_k(t)|^2\leq 
		C_\alpha T_0\int^{T_0}_0\sum^\infty_{k= N}|f^{(j+1)}_k(\tau)|^2\,d\tau\to 0$$
		as $N\to \infty$. This convergence (uniform in $t\in [0, \infty)$) implies that $\sum^\infty_{k= 0}u^{(j)}_k(t)\varphi_k$ ($j= 0, 1$) converges uniformly in $H^{2\beta}(M)$. Hence we can differentiate 
		$$u(t):=\sum^\infty_{k= 0}u_k(t)\varphi_k$$
		with respect to $t$ term-by-term. Now by the definition of Caputo derivatives (\ref{capdef}) and Proposition 23 in \cite{helin2020inverse}, we have
		$$\partial^\alpha_t u(t)=\sum^\infty_{k= 0}\partial^\alpha_t u_k(t)\varphi_k.$$
		Hence (\ref{fracODE}) ensures that $\sum^\infty_{k= 0}u_k(t)\varphi_k$ is the solution of (\ref{fracSchro}).
	\end{proof}
	
	\section{Inverse problem}
	Our goal here is to prove Theorem 1.1.
	
	Let $\lambda_{j_k}$ be the distinct eigenvalues and $P_k$ be the projections on the corresponding eigenspaces spanned by the eigenfunctions $\varphi_{j_k, p}$ ($1\leq p\leq m_k$). Here $m_k$ are the corresponding multiplicities. Then we have
	$$P_k u= \sum^{m_k}_{p=1}\langle u, \varphi_{j_k, p} \rangle \varphi_{j_k, p},\qquad u\in L^2(M).$$
	We define 
	$$P_{W_1, W_2, k}:= (P_k\circ e)|_{W_2},$$
	where $e$ is the zero extension map from $L^2(W_1)$ to $L^2(M)$. Note that
	$$P_{W_1, W_2, k}\in B(L^2(W_1)\to L^2(W_2)),$$
	which denotes the space of bounded linear operators from $L^2(W_1)$ to $L^2(W_2)$. 
	
	Due to the unique continuation property of elliptic operators, $\{\varphi_{j_k, p}|_{W_l}\}^{m_k}_{p=1}$ are linearly independent functions for $l= 1, 2$. If $$\sum^{m_k}_{p=1} c_pP_{W_1, W_2, k}(\varphi_{j_k, p}|_{W_1})= 0,$$ then it implies that
	$$\langle \sum^{m_k}_{p=1} c_pe(\varphi_{j_k, p}|_{W_1}), \varphi_{j_k, l}\rangle= 0, \qquad 1\leq l\leq m_k$$
	based on the linear independence of $\{\varphi_{j_k, p}|_{W_2}\}^{m_k}_{p=1}$.
	Combining these $m_k$ identities, it further implies $$\sum^{m_k}_{p=1} c_p\varphi_{j_k, p}|_{W_1}= 0,$$
	and thus $c_p= 0$ for $1\leq p\leq m_k$ based on the linear independence of $\{\varphi_{j_k, p}|_{W_1}\}^{m_k}_{p=1}$. Hence $P_{W_1, W_2, k}$ has rank $m_k$, and thus it is not a zero map for each $k$.
	
	We remark that based on the formula (4.9.1) in \cite{gorenflo2020mittag}, the Laplace transform formula (\ref{KkLaplace}) is valid for $s$ with $\mathrm{Re}\,s> \lambda_k^\frac{\beta}{\alpha}$. By the boundedness of $E_{\alpha, \alpha}$ on the imaginary axis, the Laplace transform on the LHS of (\ref{KkLaplace}) is well-defined for $s$ with $\mathrm{Re}\,s> 0$, so the  analytic continuation implies that (\ref{KkLaplace}) is valid for $s$ with
	$\mathrm{Re}\,s> 0$. 
	
	To prove Theorem 1.1, we will first determine the time-fractional power $\beta$ and then determine the space-fractional power $\alpha$ and the metric $g$ up to an isometry.
	
	\subsection{Determination of the time-fractional power}
	Roughly speaking, we will choose specific sources and study the asymptotic behavior of the corresponding solutions for large time $t$. The unique continuation property of the fractional Laplace-Beltrami operator will ensure that the principal part in the asymptotic expansion is non-vanishing. Then the knowledge of the source-to-solution map will enable us to determine the time-fractional power $\alpha$.
	
	More precisely, we pick a nonzero $\xi\in L^2(W_1)$ s.t. $\langle e(\xi), 1\rangle= 0$. For each $t> 0$, we define
	
	\[ g_t(\tau):=\begin{cases} 
		\tau, & 0\leq \tau\leq t/3;\\
		t/3-\tau, & t/3\leq \tau\leq 2t/3;\\
		0, & 2t/3\leq \tau\leq t.
	\end{cases}
	\]
	Note that $g_t$ is absolutely continuous and weakly differentiable, and we can choose (scalar-valued) functions $^l f(\tau)\in C^2_c((0, t))$ s.t. $^l f\to g_t$ in $H^1$ as $l\to \infty$. Let $^l u$ be the solution corresponding to the source 
	$(^l f)e(\xi)$. We use $(^l u)_k$ to denote the Fourier coefficients of $^l u$.
	Then by (\ref{ibp}) we have
	$$\lambda^\beta_k (^l u)_k(t)= 
	-(\int^t_0 E_{\alpha, 1}(i\lambda^\beta_k (t-\tau)^\alpha) (^l f)'(\tau)d\tau) \langle e(\xi), \varphi_k\rangle.$$
	
	The knowledge of the source-to-solution map $L^{\alpha, \beta}_{W_1, W_2}$ gives the knowledge
	of $(^l u)(t)|_{W_2}$ for each $l$. Let $l\to \infty$. Then we have the knowledge of 
	$w(t)|_{W_2}$. Here 
	$$w= -\sum^\infty_{k=1}\frac{1}{\lambda^\beta_k}(\int^t_0 E_{\alpha, 1}(i\lambda^\beta_k (t-\tau)^\alpha) (g_t)'(\tau)d\tau) \langle e(\xi), \varphi_k\rangle\varphi_k$$
	$$= -\sum^\infty_{k=1}\frac{1}{\lambda^\beta_k}[\int^{t/3}_0 E_{\alpha, 1}(i\lambda^\beta_k (t-\tau)^\alpha) d\tau- \int^{2t/3}_{t/3} E_{\alpha, 1}(i\lambda^\beta_k (t-\tau)^\alpha) d\tau] \langle e(\xi), \varphi_k\rangle\varphi_k.$$
	Base on the asymptotic expansion of $E_{\alpha, 1}$ in Proposition \ref{BddAsym}, for large $t$ we can further write
	$$w= C_{\alpha}t^{1-\alpha}\omega+ \mathrm{O}(t^{1-2\alpha}),\quad \omega= \sum^\infty_{k=1}\frac{1}{\lambda^{2\beta}_k}\langle e(\xi), \varphi_k\rangle\varphi_k,$$
	where $C_{\alpha}$ is a nonzero constant only depending on $\alpha$. Note that $(-\Delta_g)^{2\beta}\omega= e(\xi)$. Since $W_1$ and $W_2$ are disjoint based on the assumption in the statement of Theorem 1.1, we claim that $\omega|_{W_2}$ is nonzero. Otherwise $$(-\Delta_g)^{2\beta}\omega= \omega =0$$ in $W_2$ implies $\omega\equiv 0$ by the unique continuation property (Proposition \ref{UCP}), which contradicts with our choice of $\xi$.
	
	We can do the same analysis for $\tilde{L}^{\tilde{\alpha}, \tilde{\beta}}_{W_1, W_2}$. Then $L^{\alpha, \beta}_{W_1, W_2}= \tilde{L}^{\tilde{\alpha}, \tilde{\beta}}_{W_1, W_2}$
	implies 
	$$C_{\alpha}t^{1-\alpha}\omega|_{W_2}+ \mathrm{O}(t^{1-2\alpha})= \tilde{C}_{\tilde{\alpha}}t^{1-\tilde{\alpha}}\tilde{\omega}|_{W_2}+ \mathrm{O}(t^{1-2\tilde{\alpha}})$$
	for large $t$. We conclude that $\alpha= \tilde{\alpha}$. Otherwise we may assume $\alpha< \tilde{\alpha}$. Then we multiply both sides of the identity above by $t^{\alpha-1}$ and let $t\to\infty$ to reach the contradiction.
	
	\subsection{Determination of the space-fractional power and the metric}
	
	Now we can fix $\alpha$ since it has been determined. Next, we will retrieve the spectral information via the Laplace transform with respect to time. Weyl's law will enable us to determine the 
	space-fractional power $\beta$, and the determination result for the wave equation will enable us to determine the metric $g$ up to an isometry.
	
	\begin{prop}\label{LWHW}
		The knowledge of $L^{\alpha, \beta}_{W_1, W_2}f$ for all $f\in C^2_c((0, \infty); L^2(W_1))$
		determines the operator-valued function
		\begin{equation}\label{HWs}
			H^{\beta}_{W_1, W_2}(s):=\sum^\infty_{k=1}\frac{1}{\lambda^\beta_{j_k} + i s^\alpha }P_{W_1, W_2, k},\qquad \mathrm{Re}\,s> 0.
		\end{equation}
	\end{prop}
	\begin{proof}
		We fix a nonzero non-negative $a(t)\in C^2_c((0, \infty))$ such that $\mathcal{L}a> 0$ on $(0, \infty)$. We consider a source $f$ of the form
		$$f= a(t)e(\xi),\qquad \xi\in L^2(W_1)\quad s.t.\quad \langle e(\xi), 1\rangle= 0.$$
		(Then the Fourier coefficient $f_0 =0$, and the summation will start at $k=1$.)
		
		Recall that we use $u_f$ to denote the solution of (\ref{fracSchro}) associated with $f$. Then for each $s$ with $\mathrm{Re}\,s> 0$,
		the function $e^{-st}u_f(t)$ is $L^2$-valued integrable on $(0, \infty)$.
		By Proposition 23 in \cite{helin2020inverse}, we can take the Laplace transform componentwise. For $f= a(t)e(\xi)$, we use the convolution property of the Laplace transform to obtain
		$$\mathcal{L}u_f(s)= \sum^\infty_{k=1} \mathcal{L}(K_k* f_k)(s)\varphi_k= \sum^\infty_{k=1}\frac{1}{\lambda^\beta_k+ i s^\alpha}\mathcal{L}f_k(s)\varphi_k$$
		$$= \sum^\infty_{k=1}\frac{1}{\lambda^\beta_k+ i s^\alpha}\langle\mathcal{L}f(s), \varphi_k\rangle\varphi_k=  \mathcal{L}a(s)\sum^\infty_{k=1}\frac{1}{\lambda^\beta_{j_k}+ i s^\alpha}P_k\circ e(\xi)$$
		based on the formula for the solution of (\ref{fracSchro}) in the proof of Proposition \ref{wellpose}. Hence, we can conclude that $L^{\alpha, \beta}_{W_1, W_2}f$ determines the restriction of $\frac{1}{\mathcal{L}a(s)}\mathcal{L}u_f(s)$ in $W_2$, which equals to $H^{\beta}_{W_1, W_2}(s)\xi$.
	\end{proof}
	
	In particular, the knowledge of $H^{\beta}_{W_1, W_2}(s)$ in (\ref{HWs}) for $s>0$ determines the meromorphic operator-valued function
	\begin{equation}\label{HWz}
		H^{\beta}_{W_1, W_2}(z):= \sum^\infty_{k=1}\frac{1}{\lambda^\beta_{j_k} + i z }P_{W_1, W_2, k}
	\end{equation}
	with poles $\{i\lambda^\beta_{j_k}\}$. The convergence of the series in 
	$B(L^2(W_1)\to L^2(W_2))$-norm at non-pole $z$ is ensured by the fact that $\lambda_{j_k}\to \infty$ as $k\to \infty$ and the fact that the operators $P_k$ project onto mutually orthogonal subspaces. Note that 
	$$\lim_{z\to z_0}(\lambda^\beta_{j_k} + i z)H^{\beta}_{W_1, W_2}(z)= P_{W_1, W_2, k}$$ if $z_0= i\lambda^\beta_{j_k}$, and the limit is zero if $z_0$ is not a pole.
	
	\begin{prop}\label{PfTh1}
		Let $H^{\beta}_{W_1, W_2}(z), \tilde{H}^{\tilde\beta}_{W_1, W_2}(z)$ be the meromorphic functions associated with $g, \tilde{g}$ respectively. Suppose {$g= \tilde{g}$ in $W_1\cup W_2$} and $H^{\beta}_{W_1, W_2}(z)= \tilde{H}^{\tilde\beta}_{W_1, W_2}(z)$. Then we have $\beta= \tilde{\beta}$ and
		$$\{(\lambda_{j_k}, P_{W_1, W_2, k}): k\in \mathbb{N}\}= \{(\tilde{\lambda}_{\tilde{j}_k}, \tilde{P}_{W_1, W_2, k}): k\in \mathbb{N}\}.
		$$
	\end{prop}
	\begin{proof}
		Based on the assumption, for each $k \in \mathbb{N}$, we have
		\begin{equation}\label{HWtildeeq}
			P_{W_1, W_2, k}=
			\lim_{z\to i\lambda^\beta_{j_k}}(\lambda^\beta_{j_k} + i z)H^{\beta}_{W_1, W_2}(z) =\lim_{z\to i\lambda^\beta_{j_k}}(\lambda^\beta_{j_k} + i z)\tilde{H}^{\tilde\beta}_{W_1, W_2}(z).
		\end{equation}
		Note that the RHS of (\ref{HWtildeeq})
		equals to $\tilde{P}_{W_1, W_2, l}$, provided $\lambda^{\beta}_{j_k}= \tilde{\lambda}^{\tilde\beta}_{\tilde{j}_l}$ for some $l\in \mathbb{N}$.
		Also note that if $\lambda^{\beta}_{j_k}\notin\{\tilde{\lambda}^{\tilde\beta}_{\tilde{j}_l}: l\in \mathbb{N}\}$, then $i\lambda^\beta_{j_k}$ cannot be a pole of the $\tilde{H}^{\tilde\beta}_{W_1, W_2}(z)$ based on the previous observation, which implies the RHS of (\ref{HWtildeeq}) is zero contradicting with that $P_{W_1, W_2, k}$ is nonzero.
		
		Hence (\ref{HWtildeeq}) implies
		$$(\lambda^{\beta}_{j_k}, P_{W_1, W_2, k})\in \{(\tilde{\lambda}^{\tilde\beta}_{\tilde{j}_k}, \tilde{P}_{W_1, W_2, k}): k\in \mathbb{N}\}$$
		for each $k$, and thus
		$$\{(\lambda^{\beta}_{j_k}, P_{W_1, W_2, k}): k\in \mathbb{N}\}\subset \{(\tilde{\lambda}^{\tilde\beta}_{\tilde{j}_k}, \tilde{P}_{W_1, W_2, k}): k\in \mathbb{N}\}.
		$$			
		We can symmetrically show the reverse inclusion. Hence, we have
		$$\{(\lambda^{\beta}_{j_k}, P_{W_1, W_2, k}): k\in \mathbb{N}\}= \{(\tilde{\lambda}^{\tilde\beta}_{\tilde{j}_k}, \tilde{P}_{W_1, W_2, k}): k\in \mathbb{N}\}.
		$$			
		
		Now $\lambda^{\beta}_{j_k}= \tilde{\lambda}^{\tilde\beta}_{\tilde{j}_k}$ and
		$P_{W_1, W_2, k}= \tilde{P}_{W_1, W_2, k}$ imply that $m_k= \tilde{m}_k$ and $\lambda^{\beta}_k= \tilde{\lambda}^{\tilde\beta}_k$.
		Let $N(\lambda)$, $\tilde{N}(\lambda)$ denote the number of eigenvalues of the Laplace-Beltrami operator corresponding to $g$, $\tilde{g}$
		(counting their multiplicities) less than or equal to $\lambda$.
		Then $N(\lambda)= \tilde{N}(\lambda^\frac{\beta}{\tilde{\beta}})$. By Weyl's law,
		$$\lim_{\lambda\to \infty}\frac{N(\lambda)}{\lambda^{\frac{\mathrm{dim} M}{2}}}= C,\qquad \lim_{\lambda\to \infty}\frac{\tilde{N}(\lambda)}{\lambda^{\frac{\mathrm{dim} M}{2}}}= \tilde{C}.$$
		Hence the only possibility is $\beta= \tilde{\beta}$. Then we have
		$$\{(\lambda_{j_k}, P_{W_1, W_2, k}): k\in \mathbb{N}\}= \{(\tilde{\lambda}_{\tilde{j}_k}, \tilde{P}_{W_1, W_2, k}): k\in \mathbb{N}\}.
		$$
	\end{proof}

	The proof of Theorem 1.1 will be complete based on the following known result since the source-to-solution map associated with the wave equation below is completely characterized by the spectral data set in Proposition \ref{PfTh1}.
	
	\begin{prop}[{\cite[Theorem 1.1]{lassas2023disjoint}}]\label{classicstosol}
		Let $W_1, W_2\subset M$ be nonempty open subsets. Suppose $W_1$ satisfies the spectral bound condition (\ref{specbdcond}) and $W_2$ has smooth boundary. The metric $g$ can be determined up to an isometry from the knowledge of $g|_{W_1\cup W_2}$ and the source-to-solution map
		$$L^{\mathrm{wave}}_{W_1, W_2}: f\to u_f|_{W_2\times (0, \infty)},\qquad f\in C^\infty_c((0, \infty)\times W_1)$$
		associated with the wave equation
		\begin{equation}
			\left\{
			\begin{aligned}
				(\partial^2_{t}  -\Delta_g) u&= f,\quad \,\,\, (x, t)\in M\times (0, \infty),\\
				u(0)= \partial_t u(0)&= 0,\quad \,\,\,x\in M.\\
			\end{aligned}
			\right.
		\end{equation}
		Here $L^{\mathrm{wave}}_{W_1, W_2}$ has the representation formula
		$$L^{\mathrm{wave}}_{W_1, W_2}f(x, t)= \sum^\infty_{k=0}\int^t_0 K^{\mathrm{wave}}_{j_k}(t-\tau)P_{W_1, W_2, k}f(x, \tau)\,d\tau,$$
		where $j_0= 0$ with
		$$K^{\mathrm{wave}}_0(t)= t, \qquad
		K^{\mathrm{wave}}_k(t)= \frac{\sin(\sqrt{\lambda_k}t)}{\sqrt{\lambda_k}}\quad (k\geq 1).$$
	\end{prop}
	
	\section{Probabilistic formulation}
	Inverse problems for stochastic PDEs involving additive or multiplicative white noise have caught much attention recently (see for instance, \cite{li2017inverse, feng2023inverse}). Here we will consider a stochastic version of the deterministic problem studied in previous sections.
	
	Throughout this section, we fix a filtered probability space $(\Omega, \{\mathcal{F}_t\}, \mathbb{P})$. We use $\mathbb{E}$ to denote the expectation of a random variable, and we use $\dot {B}$ to denote the formal derivative of the standard real-valued Brownian motion $B_t$ defined on $(\Omega, \{\mathcal{F}_t\}, \mathbb{P})$.
	For $\frac{1}{2}< \alpha< 1$, we consider the fractional stochastic problem  
	\begin{equation}\label{stofracSchro}
		\left\{
		\begin{aligned}
			i\partial^\alpha_{t}u +(-\Delta_g)^\beta u &= f+ \sigma \dot {B},\quad \,\,\, \mathrm{in}\,\, M\times (0, \infty),\\
			u(0)&= 0,\quad \,\,\,\mathrm{in}\,\, M,\\
		\end{aligned}
		\right.
	\end{equation}
	where the mean of the source $f$ and the deviation $\sigma$ are deterministic functions.
	
	\begin{define}
		For $f, \sigma\in C^2_c((0, \infty); L^2(M))$, we say $u$ is the solution of (\ref{stofracSchro})
		provided that
		\begin{equation}\label{supE}
			\sup_t\mathbb{E}[\,||u(t)||^2_{L^2(M)}\,]< \infty,
		\end{equation}
		and $u$ has the form
		$\sum^\infty_{k= 0}u_k(t)\varphi_k$, where the Fourier coefficients $u_k$ solve the fractional stochastic ODE
		\begin{equation}\label{stofracODE}
			\left\{
			\begin{aligned}
				(i\partial_t^\alpha + \lambda^\beta_k) u_k&= f_k+ \sigma_k \dot {B},\\
				u_k(0)&= 0.
			\end{aligned}
			\right.
		\end{equation}
		Here $f_k(t)= \langle f(t), \varphi_k\rangle\in C^2_c((0, \infty))$, $\sigma_k(t)= \langle \sigma(t), \varphi_k\rangle\in C^2_c((0, \infty))$.
	\end{define}
	
	We remark that (\ref{stofracODE}) should be interpreted as the integral equation
	\begin{equation}\label{stofracODEint}
		u_k(t)= \frac{1}{\Gamma(\alpha)}\int^t_0 (t-\tau)^{\alpha-1}(i\lambda^\beta_k u_k(\tau)- if_k(\tau))\,\mathrm{d}\tau+ \frac{1}{\Gamma(\alpha)}\int^t_0 (t-\tau)^{\alpha-1}(- i\sigma_k(\tau))\,\mathrm{d}B_\tau.
	\end{equation}
	We will use the argument similar to the one in \cite{anh2019variation} to derive the formula for $u_k$ based on the knowledge of the corresponding deterministic problem and some basic facts on It\^o integrals.
	
	\begin{prop}
		(\ref{stofracSchro}) has the solution
		$u= \sum^\infty_{k= 0}u_k(t)\varphi_k$, where 
		\begin{equation}\label{stofracODEuk}
			u_k(t)= \int^t_0 -i(t-\tau)^{\alpha-1}
			E_{\alpha, \alpha}(i\lambda^\beta_k (t-\tau)^\alpha)f_k(\tau)\,\mathrm{d}\tau
		\end{equation}
		$$+\int^t_0 -i(t-\tau)^{\alpha-1}
		E_{\alpha, \alpha}(i\lambda^\beta_k (t-\tau)^\alpha)\sigma_k(\tau)\,\mathrm{d}B_\tau.$$
	\end{prop}
	\begin{proof}
		We use $v_k(t)$ to denote the RHS of (\ref{stofracODEuk}). We will first show that 
		$u_k(t)= v_k(t)$. Note that by It\^o's representation theorem (see Theorem 4.3.3 in \cite{oksendal2013stochastic}), each $X\in L^2(\mathcal{F}_t, \mathbb{P})$ can be written as 
		$$X= \mathbb{E}X+ \int^t_0 \xi(\tau)\,\mathrm{d}B_\tau$$
		for some $\mathcal{F}_t$-adapted stochastic process $\xi$ s.t. 
		$$\mathbb{E}[\int^t_0|\xi(\tau)|^2\mathrm{d}\tau]
		< \infty.$$
		Hence, it suffices to show that 
		\begin{equation}\label{Euvid}
			\mathbb{E}[u_k(t)(c+\int^t_0 \xi(\tau)\,\mathrm{d}B_\tau)]=
			\mathbb{E}[v_k(t)(c+\int^t_0 \xi(\tau)\,\mathrm{d}B_\tau)]
		\end{equation}
		for each constant $c$ and real-valued process $\xi$. 
		
		Now we use $\chi_{\xi, c, k}(t)$ and $\tilde{\chi}_{\xi, c, k}(t)$ to denote the LHS and RHS of (\ref{Euvid}) respectively, and we define
		$$\kappa_{\xi, c, k}(t):= \mathbb{E}[f_k(t)(c+\int^t_0 \xi(\tau)\,\mathrm{d}B_\tau)].$$ 
		In fact, we can multiply both sides of (\ref{stofracODEint}) by $c+\int^t_0 \xi(\tau)\,\mathrm{d}B_\tau$ and take the expectation to obtain 
		$$\chi_{\xi, c, k}(t)= \frac{1}{\Gamma(\alpha)}\int^t_0 (t-\tau)^{\alpha-1}(i\lambda^\beta_k \chi_{\xi, c, k}(\tau)- i \kappa_{\xi, c, k}(\tau))\,\mathrm{d}\tau$$
		$$+ \frac{-i}{\Gamma(\alpha)}\mathbb{E}[(\int^t_0 (t-\tau)^{\alpha-1}\sigma_k(\tau)\,\mathrm{d}B_\tau)(\int^t_0 \xi(\tau)\,\mathrm{d}B_\tau)]$$
		$$= \frac{1}{\Gamma(\alpha)}\int^t_0 (t-\tau)^{\alpha-1}(i\lambda^\beta_k \chi_{\xi, c, k}(\tau)- i \kappa_{\xi, c, k}(\tau))\,\mathrm{d}\tau+\frac{-i}{\Gamma(\alpha)}(\int^t_0 (t-\tau)^{\alpha-1}\sigma_k(\tau) \mathbb{E}[\xi(\tau)]\,\mathrm{d}\tau,$$
		where the last identity above is ensured by It\^o's isometry. This implies that 
		$\chi_{\xi, c, k}$ solves the deterministic fractional ODE 
		\begin{equation}\label{ChifracODE}
			\left\{
			\begin{aligned}
				(i\partial_t^\alpha + \lambda^\beta_k) \chi_{\xi, c, k}(t)&= \kappa_{\xi, c, k}(t)+
				\sigma_k(t) \mathbb{E}[\xi(t)],\\
				\chi_{\xi, c, k}(0)&= 0.
			\end{aligned}
			\right.
		\end{equation}
		
		On the other hand, we can multiply the expression of $v_k(t)$ by $c+\int^t_0 \xi(\tau)\,\mathrm{d}B_\tau$. Again we take the expectation and use It\^o's isometry to obtain 
		$$\tilde{\chi}_{\xi, c, k}(t)= \int^t_0 -i(t-\tau)^{\alpha-1}
		E_{\alpha, \alpha}(i\lambda^\beta_k (t-\tau)^\alpha)\kappa_{\xi, c, k}(\tau)\,\mathrm{d}\tau$$
		$$+\mathbb{E}[(\int^t_0 -i(t-\tau)^{\alpha-1}
		E_{\alpha, \alpha}(i\lambda^\beta_k (t-\tau)^\alpha)\sigma_k(\tau)\,\mathrm{d}B_\tau)(\int^t_0 \xi(\tau)\,\mathrm{d}B_\tau)]$$
		$$= \int^t_0 -i(t-\tau)^{\alpha-1}
		E_{\alpha, \alpha}(i\lambda^\beta_k (t-\tau)^\alpha)\kappa_{\xi, c, k}(\tau)\,\mathrm{d}\tau$$
		$$+ \int^t_0 -i(t-\tau)^{\alpha-1}
		E_{\alpha, \alpha}(i\lambda^\beta_k (t-\tau)^\alpha)\sigma_k(\tau)\mathbb{E}[\xi(\tau)]\,\mathrm{d}\tau.
		$$
		Based on (\ref{fracODE}, \ref{Kfconvolute}, \ref{Kk}) in Proposition 3.1, we know that $\tilde{\chi}_{\xi, c, k}$ is the solution of
		(\ref{ChifracODE}) as well. Hence, we conclude that ${\chi}_{\xi, c, k}= \tilde{\chi}_{\xi, c, k}$.
		
		To verify the estimate (\ref{supE}), we use $u^D_k(t)$ and $u^I_k(t)$ to denote the deterministic integral and the It\^o integral in (\ref{stofracODEuk}). Suppose $\mathrm{supp}\,f\cup \mathrm{supp}\,\sigma \subset (0, T_0)\times M$ for some $T_0> 0$.
		By It\^o's isometry, the boundedness of $E_{\alpha, \alpha}$ on the imaginary axis and Cauchy-Schwarz inequality, we have
		$$\mathbb{E}[\,||u(t)||^2_{L^2(M)}\,]=
		\sum^\infty_{k=0}\mathbb{E}[|u_k(t)|^2]
		= \sum^\infty_{k=0}|u^D_k(t)|^2+
		\sum^\infty_{k=0}\mathbb{E}[|u^I_k(t)|^2]
		$$
		$$\leq CT^{2\alpha-1}_0\sum^\infty_{k=0}\int^{T_0}_0|f_k(\tau)|^2\,\mathrm{d}\tau+ \sum^\infty_{k=0}\int^t_0(t- \tau)^{2(\alpha-1)}|E_{\alpha, \alpha}(i\lambda^\beta_k (t-\tau)^\alpha)|^2|\sigma_k(\tau)|^2\,\mathrm{d}\tau$$
		$$\leq CT^{2\alpha-1}_0\int^{T_0}_0||f(\tau)||^2_{L^2(M)}\,\mathrm{d}\tau+ C'\int^t_0(t- \tau)^{2(\alpha-1)}||\sigma(\tau)||_{L^2(M)}^2\,\mathrm{d}\tau$$
		$$\leq CT^{2\alpha-1}_0\int^{T_0}_0||f(\tau)||^2_{L^2(M)}\,\mathrm{d}\tau+ C''T^{2\alpha-1}_0\sup_t ||\sigma(t)||_{L^2(M)}^2< \infty.$$
	\end{proof}
	
	Note that $\mathbb{E}[u^I_k(t)]= 0$ so the expectation of the solution of the stochastic problem (\ref{stofracSchro}) is exactly the same as the solution of the deterministic problem (\ref{fracSchro}). Hence, we can formulate the inverse problem and state Theorem 1.1 in the following probabilistic way.
	
	For given nonempty open subsets $W_1, W_2 \subset M$ and fixed deterministic deviation function $\sigma$, we define the source-to-solution map
	\begin{equation}\label{stosol2}
		L^{\alpha, \beta}_{W_1, W_2}: f\to u_f|_{W_2\times (0, \infty)}, \qquad \mathrm{supp} (f)\subset W_1\times (0, \infty),
	\end{equation}
	where the stochastic function $u_f$ is the solution of (\ref{stofracSchro}) corresponding to the mean of the source $f$. 
	
	\begin{thm}\label{Th2}
		Suppose $(M, g)$, $(M, \tilde{g})$ are smooth connected closed Riemannian manifolds. Suppose $\frac{1}{2}< \alpha, \tilde{\alpha}< 1$ and $0< \beta,  \tilde{\beta}< 1$. Let $W_1, W_2\subset M$ be nonempty disjoint open subsets. Suppose $W_1$ satisfies the spectral bound condition (\ref{specbdcond}) and $W_2$ has smooth boundary. We fix the deterministic function $\sigma\in C^2_c((0, \infty); L^2(W_1))$. Let $L^{\alpha, \beta}_{W_1, W_2}, \tilde{L}^{\tilde{\alpha}, \tilde{\beta}}_{W_1, W_2}$ be the source-to-solution maps corresponding to $g, \tilde{g}$.
		Suppose {$g= \tilde{g}$ in $W_1\cup W_2$} and
		\begin{equation}\label{Estosoleq}
			\mathbb{E}[L^{\alpha, \beta}_{W_1, W_2} f]= \mathbb{E}[\tilde{L}^{\tilde{\alpha}, \tilde{\beta}}_{W_1, W_2} f],\qquad
			\text{for }f\in C^2_c((0, \infty); L^2(W_1)).
		\end{equation}
		Then $\alpha= \tilde{\alpha}$, $\beta=  \tilde{\beta}$, and $(M, g)$ and $(M, \tilde{g})$ are isometric Riemannian manifolds.
	\end{thm}
	
	\bibliographystyle{plain}
	{\small\bibliography{Reference12}}
\end{document}